\documentclass[12pt]{amsart}

\usepackage{amsmath,amssymb,amsthm,verbatim}
\usepackage{hyperref}

\newtheorem{theorem}[subsection]{Theorem}
\newtheorem{lemma}[subsection]{Lemma}

\theoremstyle{definition}

\newtheorem{remark}[subsection]{Remark}
\newtheorem{example}[subsection]{Example}

\newcommand{\dist}{\mathrm{dist}}

\newcommand{\haus}{\mathcal{H}}

\newcommand{\spt}{\mathrm{spt}}

\newcommand{\reg}{\mathrm{reg}}

\newcommand{\eps}{\epsilon}

\newcommand{\sing}{\mathrm{sing}}

\newcommand{\R}{\mathbb{R}}

\newcommand{\del}{\partial}

\newcommand{\dom}{\mathrm{dom}}

\newcommand{\cI}{\mathcal{I}}

\title{A global bound for the singular set of area-minimizing hypersurfaces}
\author{Nick Edelen}
\address {Massachusetts Institute of Technology, 77 Massachusetts Avenue, Cambridge, MA 02139, USA}
\email{edelen@mit.edu}

\begin{document}

\begin{abstract}
We give an a priori bound on the $(n-7)$-dimensional measure of the singular set for an area-minimizing $n$-dimensional hypersurface, in terms of the geometry of its boundary.
\end{abstract}

\maketitle


Area-minimizing surfaces in general will not be smooth, and a basic question in minimal surface theory is to understand the size and nature of the singular region.  The cumulative works of many (Federer, De Giorgi, Allard, Simons, to name only a few) prove that for absolutely-area-minimizing $n$-dimensional hypersurfaces in $\R^{n+1}$ (``codimension-one area-minimizing integral currents''), the interior singular set is at most $(n-7)$-dimensional.  This dimension bound is sharp, and is directly tied to the existence of low-dimensional, non-flat minimizing cones.

\cite{hardt-simon:boundary} proved that for such codimension-one area-minimizers, if the boundary is known to be $C^{1,\alpha}$ and multiplicity-one, then in fact no singularities lie within a neighborhood of the boundary.  Combined with interior regularity, this theorem gives a very nice structure of these minimizing hypersurfaces.

Recently \cite{naber-valtorta}, \cite{naber-valtorta:varifold} quantified the interior partial regularity, by demonstrating effective local (interior) bounds on the $\haus^{n-7}$ measure of the singular set.  Their methods also prove $(n-7)$-rectifiability of the singular set, which was originally established through an entirely different approach by \cite{simon:rect}.

In this short note, we obtain obtain a global, \emph{effective} a priori estimate on the singular set of an area-minimzing hypersurface in terms of the boundary geometry.  Our results are loosely analogous to the a priori bounds of \cite{almgren-lieb} (see also the recent works \cite{misk:3d-sing}, \cite{misk:nd-sing}).

\vspace{5mm}

We work in $\R^{n+1}$, for $n \geq 7$.  Let us write $\cI_n(U)$ for the space of integral $n$-currents acting on forms supported in the open set $U$.  Given an $n$-dimensional, oriented manifold $E$, write $[E]$ for the current induced by integration.  Let $\eta_\lambda(x) = \lambda x$, and $\tau_y(x) = x + y$.

If $T \in \cI_n(U)$, we say $T$ is area-minimizing if $||T||(W) \leq ||T+S||(W)$ for every open $W \subset\subset U$, and every $S \in \cI_n(U)$ satisfying $\del S = 0$, $\spt S \subset W$.  The regular set $\reg T$ is the (open) set of points where $\spt T$ is locally the union of embedded $C^{1,\alpha}$ manifolds.  The singular set is $\sing T = \spt T \setminus \reg T$.  Write $||T||$ for the mass measure of $T$.

Given an $k$-manifold $S$, and $x \in S$, let $r_{1,\alpha}(S, x)$ be the largest radius $r$, so that $(S - x)/r$ is the graph of a $C^{1,\alpha}$ function $u$, with $|u|_{1,\alpha} \leq 1$.  Define $r_{1,\alpha}(S) = \inf_{x \in S} r_{1,\alpha}(S, x)$.

Our main Theorem is the following.
\begin{theorem}\label{thm:main}
There is a constant $c = c(n, \alpha)$ so that the following holds.  Let $T$ be a area-minimizing integral $n$-current in $\R^{n+1}$.  Suppose $\del T$ is a multiplicity-one, compact, oriented $C^{1,\alpha}$ manifold $S$, and assume that $S$ is contained in the boundary of some convex set.  Then
\begin{gather}
\haus^{n-7}(\sing T) \leq c(n, \alpha) \frac{ ||T||(\R^{n+1}) }{r_{1,\alpha}(S)^7}
\end{gather}
In particular, we have
\[
\haus^{n-7}(\sing T) \leq c'(n,\alpha)\frac{\haus^{n-1}(S)^{\frac{n}{n-1}}}{r_{1,\alpha}(S)^7}.
\]
\end{theorem}

I believe Theorem \ref{thm:main} should hold for more general $S$, but there are subtleties even in the idealized case when $S$ is a line.  See the discussion below.


We also have a version of Theorem \ref{thm:main} in the case when $T$ has free-boundary.  Given open sets $U$, $\Omega$, we say $T \in \cI_n(U)$ is area-minimizing with free-boundary in $\Omega$ if: $\spt T \subset \overline{\Omega}$, and $||T||(W) \leq ||S + T||(W)$ for all $W \subset\subset U$, and every $S \in \cI_n(U)$ satisfying $\spt S \subset \overline{\Omega} \cap W$ and $\spt(\del S) \subset \del \Omega$.  \cite{gruter:optimal-reg} proved boundary singularities have dimension at most $n-7$.
\begin{theorem}\label{thm:fb}
Let $\Omega$ be a domain with $C^{2}$-boundary, and $\infty > r_{1,1}(\del \Omega) > 0$.  Let $T$ be a compactly supported, area-minimizing current with free-boundary in $\Omega$, with $\del T \llcorner \Omega = 0$.  Then
\[
\haus^{n-7}(\sing T) \leq c(n) \frac{||T||(\Omega)}{r_{1,1}(\del\Omega)^7}.
\]
\end{theorem}


The key to proving both Theorems is the observation that Naber-Valtorta's technique gives the following \emph{linear} interior bound on the singular set: if $T$ is area-minimizing in $U \subset \R^{n+1}$, with $\del T \llcorner U = 0$, then for every $\eps > 0$, we have:
\[
\haus^{n-7}(\sing T \cap U \setminus B_\eps(\del U)) \leq c(n) \eps^{-7} ||T||(U \setminus B_{\eps/2}(\del U)) .
\]  

For the Neumann problem (Theorem \ref{thm:fb}), we can adapt the techniques of \cite{naber-valtorta:varifold} to prove a priori estimates on the singular set in a neighborhood of the barrier.  Unfortunately, it's not clear that a good Dirichlet boundary version of Naber-Valtorta exists, in any more generality than is considered in Theorem \ref{thm:main}.  The problem is that there is not necessarily a good relationship between regularity and symmetry.  If there exists a singular, minimizing hypersurface with Euclidean area growth and linear boundary, then by \cite{hardt-simon:boundary} any blow-down sequence would preclude an inclusion like $\sing T \subset S^{n-7}_\eps$ (here $S^{n-7}_\eps$ being the $(n-7, \eps)$-strata of \cite{cheeger-naber:ricci}).

Instead, for Theorem \ref{thm:main}, we can prove an effective version of \cite{hardt-simon:boundary}, which says that the singular set is some uniform distance away from the boundary curve.  It's tempting to think an \emph{ineffective}, quantitative version of \cite{hardt-simon:boundary} might hold for more general Dirichlet setups, but the problem is the same as above.


\begin{remark}
The following variant of Theorem \ref{thm:main} holds for almost-area-minimizers.  Let $T \in \cI_n(\R^{n+1})$ be almost-area-minimizing, in the sense that
\[
||T||(B_r(x)) \leq ||T + S||(B_r(x)) + c_0 r^{n+2\alpha},
\]
for any $S \in \cI_n(\R^{n+1})$, $\del S = 0$, $\spt S \subset B_r(x)$, and some fixed $c_0$.  Suppose $\del T = [S]$ is an oriented, embedded, multiplicity-one $C^{1,\alpha}$-manifold $S$, and suppose there is a $C^{1,\alpha}$ domain $\Omega$ so that $\spt T \subset \overline{\Omega}$, $S \subset \del\Omega$.  Then
\[
\haus^{n-7}(\sing T) \leq c(n,\alpha) \max\{ c_0^{\frac{7}{2\alpha}}, r_{1,\alpha}(\del\Omega)^{-7}, r_{1,\alpha}(S)^{-7} \} ||T||(\R^{n+1}) .
\]
The same proof works, using \cite{duzaar-steffan}, \cite{bombieri:almost-minz} in place of \cite{hardt-simon:boundary}, \cite{allard:first-variation}, and a minor modification of \cite{naber-valtorta:varifold}.
\end{remark}

The following examples illustrates some of the problems in extending our proof of Theorem \ref{thm:main} to more general settings.
\begin{example}\label{examples}
Both the half-helicoid and half of Enneper's surface (\cite{white:enneper}, \cite{perez:enneper}) are area-minimizing $2$-dimensional currents in $\R^3$.  (For the half-helicoid, just observe that by rotating the half-helicoid about the $z$-axis, one obtains a smooth foliation of $\R^3 \setminus \text{z-axis}$ by oriented minimal surfaces). It would be interesting to know if there exists an example of a singular minimizing hypersurface bounding a multiplicity-one line.

The half-helicoid structure could be seen locally for finite $S$, if one does not assume a priori area bounds on $S$.  For example, one can imagine a connected boundary curve $S$, which is composed of line segment $L$, and a curve that wraps around $L$ many times.  By taking the wrapping curve to go further and further out, one can arrange $S$ to satisfy $r_{1,\alpha}(S) \geq 1$, but take the separation along $L$ of the wrappings to zero.  The minimizing integral current $T$ spanning $S$ will look very much like a compressed half-helicoid near the line segment.  We cannot decompose this $T$ near $L$ into pieces of uniformly bounded area.

\end{example}

I thank Otis Chodosh for several illuminating conversations, and pointing out the half-helicoid is area-minimizing.  This work was supported in part by NSF grant DMS-1606492.



\subsection{Proof of Theorem \ref{thm:main}}

The following quantifies Hardt-Simon's boundary regularity.
\begin{lemma}\label{lem:quant-hs}
There is a constant $\eps_1(n, \alpha)$, so that the following holds.  Let $T$ and $S$ be as in Theorem \ref{thm:main}.  Then for all $x \in \sing T$, we have 
\begin{gather}
\inf_{y \in S} \frac{|x - y|}{r_{1,\alpha}(S, y)} \geq \eps_1(n, \alpha).
\end{gather}
\end{lemma}

\begin{proof}
Towards a contradiction, suppose we have minimizing currents $T_i$, with boundary curves $S_i$, each contained the boundary of the convex set $\Omega_i$, and $x_i \in \sing T_i$, and $y_i \in S_i$, so that
\[
\dist(x_i, S_i) \leq r_{1,\alpha}(S, y_i)/i.
\]
By the maximum principle, $\spt T_i \subset \overline{\Omega_i}$, and by \cite{hardt-simon:boundary}, $\dist(x_i, S_i) > 0$.  Since $r_{1,\alpha}(S, y) \leq r_{1,\alpha}(S, y_i)/2$ for $y \in B_{r_{1,\alpha}(S, y)/2}(y)$, there is no loss in generality in assuming that $y_i$ realizes the distance in $S$ to $x_i$.

After a rotation, dilation, translation, we can assume $y_i = 0$, $r_{1,\alpha}(S_i, y_i) = 2$, and $e_1$ is a choice of vector so that $\Omega_i \subset \{ x : x \cdot e_1 < 0 \}$.  Moreover, we can take $S_i \cap B_1$ to be the graph of a function $u_i$, define on the line $L = \{ x_1 = x_{n+1} = 0 \}$, with $|u_i|_{1,\alpha} \leq 1$.  Notice that $\dom(u_i) \supset B_{1/2} \cap L$, and that $x_i \to 0$.  Let us assume $x_i \in B_{1/2}$ for all $i$.

Let $h(t, x)  : [-1, 1] \times (L \cap B_{1/2}) \to \R^{n+1}$ be defined as
\[
h(t, x) = \left\{ \begin{array}{l l} x + t u_i(x) & t \geq 0 \\ x - t \sqrt{1-|x|^2} e_1 & t \leq 0 \end{array} \right. , 
\]
and let $R_i = (h_i)_\sharp([-1, 1] \times [L \cap B_{1/2}])$.  Then, as an element of $\cI_n(B_{1/2})$, $\del R_i = [S_i] \llcorner B_{1/2}$.  In particular, $T_i - R_i \in \cI_n(B_{1/2})$ has no boundary.  By standard decomposition of codimension-one currents \cite{simon:gmt}, we can find open sets $E_{i,j} \subset E_{i, j+1} \subset \ldots B_{1/2}$, so that $[E_{i,j}] \in \cI_n(B_{1/2})$ satisfies:
\[
T_i - R_i = \sum_j \del [E_{i,j}], \quad ||T_i - R_i|| = \sum_i ||\del [E_{i,j}]||.
\]

Since $(T_i - R_i) \llcorner \Omega^c = (-R_i) \llcorner \Omega^c$, and the $E_{i,j}$ are nested, we have $\spt \del[E_{i,j}] \subset \overline{\Omega}$ for all but one $j = j_i$.  Therefore, we have
\[
||T_i|| = ||\del [E_{i,j_i}] + R_i|| + \sum_{j \neq j_i} ||\del [E_{i,j}]||,
\]
and hence $\del [E_{i,j_i}] + R_i$ and each $\del[E_{i,j}]$ ($j\neq j_i$) are area-minimizing.  By volume comparison agains balls, and the estimate $||R_i||(B_r) \leq c(n) r^n$, we get that
\[
||\del [E_{i,j}]||(B_r) \leq c(n) r^n \quad \forall r < 1/4.
\]

We break into two cases.  First, assume that $x_i \in \spt(\del[E_{i,j_1}] + R_i)$ for all $i$.  Let $\lambda_i = |x_i|^{-1}$, and consider the dilates
\[
T'_i := \del [(\eta_{\lambda_i})_\sharp E_{i, j_1}] + (\eta_{\lambda_i})_\sharp R_i .
\]
So that $T'_i$ has a singularity at distance $1$ from $\lambda_i S_i$.

We can pass to a subsequence (also denoted $i$), so that $(\eta_{\lambda_i})_\sharp [E_{i, j_1}] \to [E]$, for some open set $E$.  Since $\lambda_i S_i \to L$ in $C^{1,\alpha}$, we have $(\eta_{\lambda_i})_\sharp R_i \to [H]$, where $H = \{ x_{n+1} = 0, x_1 > 0 \}$ and $[H]$ is endowed with the orientation so that $\del [H] = L$.

In particular, we have $T'_i \to T = \del [E] + [H]$, where $\del T = [L]$.  Since each $T'_i$ is minimizing, $T$ is minimizing also, and $T_i'$ converge as both currents and measures.  By construction, $T$ has a singularity at distance $1$ from $L$, $T$ has Euclidean volume growth, and $\spt T \subset \{ x : x \cdot e_1 \leq 0\}$.

Since $T$ is minimizing with Euclidean volume growth, we can take a tangent cone $C$ at infinity (as both currents and variflds).  $C$ satisfies $\del [C] = [L]$, and so by \cite{hardt-simon:boundary} $C$ is planar.  Since we can write $C = \del [F] + [H]$ for some open set $F$, and $\spt C \subset \{ x : x \cdot e_1 \leq 0\}$, in fact $C$ must be a multiplicity-one half-plane.  By monotonicity we must have that $T$ is a multiplicity-one half-plane also, and hence $T$ is regular.  This is a contradiction.



We are left with the second case: for all $i$,  $x_i \in \spt \del[E_{i, j}]$ for some $j \neq j_i$.  Write $E_i = E_{i, j}$ for the open set, for which $x_i \in \spt \del E_i$.  Consider the dilates $E'_i = \lambda_i E_i$.  Then we can pass to a subsequence, to get convergence as currents $[E'_i] \to [E]$, convergence as currents and measures $\del [E'_i] \to \del [E]$, for $0 \neq [E] \in \cI_{n+1}(\R^{n+1})$ satisfying: a) $\del [E]$ is minimizing; b) $\del[E]$ has a singularity at distance $1$ from the origin; and c) $E \subset \{ x : x \cdot e_1 \leq 0\}$.

Properties a), c) imply that any tangent cone at infinity of $\del[E]$ is a multiplicity-one plane, and hence $\del[E]$ is a multiplicity-one plane.  This contradicts property b), and therefore completes the proof of Lemma \ref{lem:quant-hs}.
\end{proof}

\begin{lemma}\label{lem:local-bound}
Let $T \in \cI_n(B_1)$ be area-minimizing, with $\del T = 0$.  Then we have
\[
\haus^{n-7}(\sing T \cap B_{1/2}) \leq c(n) ||T||(B_{1}).
\]
\end{lemma}

\begin{proof}
We can decompose $T = \sum_i \del [E_i]$, for $E_i \subset E_{i+1} \subset\ldots \subset B_1$, so that $||T|| = \sum_i ||\del [E_i]||$, and hence each $\del [E_i] \in \cI_n(B_1)$ is minimizing in $B_1$.

Since $||\del [E_i]||(B_{1}) \leq ||\del [E_i \cup B_{3/4}]||(B_1)$, we have $||\del [E_i]||(B_{3/4}) \leq c(n)$.  On the other hand, by monotonicity,  if $\spt \partial [E_i] \cap B_{1/2} \neq \emptyset$, then $||\partial [E_i]||(B_{1}) \geq 1/c(n)$.  From the estimates of \cite{naber-valtorta:varifold}, we have
\[
\haus^{n-7}(\sing \partial [E_i] \cap B_{1/2}) \leq c(n) \leq c(n)||\partial [E_i]||(B_{1}).
\]

We can sum up contributions:
\begin{align*}
\haus^{n-7}(\sing T \cap B_{1/2})
&\leq \sum_i \haus^{n-7}(\sing \partial [E_i] \cap B_{1/2}) \\
&\leq c(n) \sum_i ||\partial [E_i]||(B_{1}) \\
&= c(n) ||T||(B_{1}). \qedhere
\end{align*}
\end{proof}

\begin{proof}[Proof of Theorem \ref{thm:main}]
By scaling, there is no loss in assuming $r_{1,\alpha}(S) = 1$.  Lemma \ref{lem:quant-hs} implies that $B_\eps(S) \cap \sing T = \emptyset$, where $\eps = \eps_1(n, \alpha)$.

Let $\{x_j\}_j$ be a maximal $(\eps/2)$-net in $\spt T \setminus B_\eps(S)$.  Then the balls $\{B_{\eps/2}(x_j)\}_j$ cover $\spt T \setminus B_\eps(S)$, and the balls $\{B_{\eps}(x_j)\}_j$ have overlap bounded by $c(n)$.  For each $j$, $\partial T \llcorner B_\eps(x_j) = 0$, and so by Lemma \ref{lem:local-bound} we have
\[
\haus^{n-7}(\sing T \cap B_{\eps/2}(x_j)) \leq \frac{c(n)}{\eps^7} ||T||(B_\eps(x_j)).
\]

Using bounded overlap of the $\{B_{\eps}(x_j)\}_j$, and the isoperimetric inequality due to \cite{almgren:iso}, we deduce that
\begin{align*}
\haus^{n-7}(\sing T) &= \haus^{n-7}(\sing T \setminus B_\eps(S)) \\
&\leq \sum_j \haus^{n-7}(\sing T \cap B_{\eps/2}(x_j)) \\
&\leq c(n, \alpha) \sum_j ||T||(B_\eps(x_j)) \\
&\leq c(n, \alpha) ||T||(\R^{n+1}) \\
&\leq c(n, \alpha) \haus^{n-1}(S)^{n/(n-1)}. \qedhere
\end{align*}
\end{proof}

\subsection{Proof of Theorem \ref{thm:fb}}\label{sec:variants}

We will show that the arguments of \cite{naber-valtorta:varifold}, \cite{gruter:optimal-reg}, and \cite{gruter-jost:allard} prove the following: there is an $\eps = \eps(n)$, so that for $x \in \spt T \cap \del\Omega$, and $r = r_{1,1}(\del\Omega)$, we have
\begin{equation}\label{eqn:fb-b-est}
\haus^{n-7}(\sing T \cap B_{\eps r/2}(x)) \leq c(n) ||T||(B_{\eps r}(x))
\end{equation}
Given this estimate, the bound of Theorem \ref{thm:fb} follows by a straightforward covering argument as in the proof of Theorem \ref{thm:main}.

\vspace{5mm}

By scaling, we can and shall assume that $r_{1,1}(\del\Omega) = 1/\Gamma$, for $\Gamma \leq \eps_2(n)$ chosen sufficiently small so that in $B_{4}(\del\Omega)$ the nearest-point projection $\xi(x)$ to $\del\Omega$ is well-defined and satisfies $|\xi|_{C^1} \leq 1$.  Define the reflection function $\sigma(x) = 2\xi(x) - x$, and the linear reflection $i_x$ about $T_{\xi(x)} \del\Omega$.


Take $T \in \cI_n(B_2)$ area-minimizing with free-boundary in $\Omega$.  Define $T' = T - \sigma_\sharp T$, so that $\del T' = 0$.  Then we can decompose $T'$ as
\[
T' = \sum_i \del[E_i], \quad ||T'|| = \sum_i ||\del [E_i]||,
\]
for nested open sets $E_i \subset E_{i+1}$.  Moreover, since $T' \llcorner \Omega = T$ we can write
\begin{gather}\label{eqn:fb-inside-decomp}
T = \sum_i \del[E_i] \llcorner \Omega, \quad ||T|| = \sum_i ||\del [E_i] \llcorner\Omega||.
\end{gather}
From \eqref{eqn:fb-inside-decomp}, we get that each $\del [E_i] \llcorner \Omega$ is area-minimizing, with free-boundary in $\Omega$.  By comparison against $\del [E_i \cup B_r(x)]\llcorner \Omega$, we have the a priori mass bounds
\[
||\del[E_i]\llcorner \Omega||(B_r(x)) \leq c(n) r^n \quad \forall B_r(x) \subset\subset B_2.
\]
Additionally, \cite{gruter:optimal-reg} showed $T'$ admits a certain almost-minimizing property, in the following sense:
\begin{gather}\label{eqn:fb-almost-minz}
||T'||(B_r(x)) \leq ||T' + S||(B_r(x)) + c(n) \Gamma r ||T'||(B_r(x)).
\end{gather}
for every $S \in \cI_n(B_2)$ with $\del S = 0$, $\spt S \subset B_r(x)$, and every $B_r(x) \subset B_2$ with $x \in \del\Omega$.

\cite{gruter-jost:allard} define the following monotonicity.  For $x \in B_{1}$, and $r < 1-|x|$, let
\[
\tilde \theta_T(x, r) = r^{-n} ||T||(B_r(x)) + r^{-n} ||T|| (\{ y : |\sigma(y) - x| < r \}) .
\]
Notice that when $\Omega$ is a half-space, then $\tilde \theta_T(x, r) = \theta_{T'}(x, r)$, and in general we have $\tilde \theta_T(x, r) = \theta_T(x, r)$ when $r < \dist(x, \del\Omega)$.  Here $\theta_T(x, r) = r^{-n}||T||(B_r(x))$ for the usual Euclidean density ratio, and $\theta_T(x) = \lim_{r\to 0} \theta_T(x, r)$ whenever it exists.

For $0 < s < r < 1-|x|$, \cite{gruter-jost:allard} prove
\begin{align}\label{eqn:fb-mono}
&\int_{B_r(x) \setminus B_s(x)}  |y - x|^{-n-2} \left( |(y - x)^\perp|^2 + |i_{y}(\tilde y - x)^\perp|^2\right) d||T||(y) \\
&\quad \leq \tilde \theta_T(x, r) - \tilde \theta_T(x, s) + c(n) \Gamma r \tilde\theta_T(n, r). \nonumber
\end{align}
Monotonicity \eqref{eqn:fb-mono} implies that the density $\tilde\theta_T(x) = \lim_{r \to 0} \tilde\theta_T(x, r)$ is a well-defined, upper-semi-continuous function on $B_1$, which is $\geq 1$ on $\spt T$.

The above discussion, and the works of \cite{gruter:optimal-reg}, \cite{gruter-jost:allard}, give:
\begin{lemma}\label{lem:fb-conv}
Let $\Omega_i$ be a sequence of $C^2$ domains, with $r_{1,1}(\del\Omega_i \cap B_2) \to \infty$, and $T_i \in \cI_n(B_2)$ a sequence of area-minimizing currents with free-boundary in $\Omega_i$.  Suppose $T_i \to T$.  Then
\begin{enumerate}
\item $T$ is area-minimizing, with free-boundary in a half-space, and $||T_i|| \to ||T||$.
\item $T'_i \to T'$ as currents and measures, and $\tilde\theta_{T_i}(x, r) \to \theta_{T'}(x, r)$ for all $x \in B_2$, and a.e. $0 < r < 2-|x|$.  Here $T_i' = T_i - (\sigma_i)_\sharp T_i$, where $\sigma_i$ is the reflection function associated to $\Omega_i$.
\item If $x_i \to x \in B_2$, and $r_i \to 0$, then $\limsup_i \tilde\theta_{T_i}(x_i, r_i) \leq \tilde\theta_T(x) = \theta_{T'}(x)$.
\item If $T'$ is regular, then the $T_i' \llcorner B_1$ are regular for $i$ sufficiently large.
\end{enumerate}
\end{lemma}

\begin{proof}
Since $\sigma_i \to \sigma$ in $C^{1}$, we have $T_i' \to T'$.  The convergence of measures $||T_i'|| \to ||T'||$ is a standard argument using the almost-minimizing property \eqref{eqn:fb-mono}.  Convergence $||T_i|| \to ||T||$ then follows from the fact that $T_i' \llcorner \overline{\Omega_i} = T_i$.

Convergence of the $\tilde\theta_T$ follows because we can estimate
\begin{align*}
&\Big| ||T_i||(\{ y : |\sigma_i(y) - x| < r \}) - ||T_i||(\sigma(B_r(x)) \Big| \\
&\quad \leq c ||T_i||(B_{(1+\kappa_i)r}(x) \setminus B_{(1-\kappa_i)r}(x)),
\end{align*}
where $\kappa_i \to 0$ as $i \to \infty$, and because $||T||(\del B_r(x)) = 0$ for a.e. $r$.  Upper-semi-continuity follows by convergence of $\tilde\theta_T(x, r)$, and monotonicity.

The last property (4) is a direct consquence of the decomposition \eqref{eqn:fb-inside-decomp} and the Allard-type regularity theory of \cite{gruter-jost:allard}.
\end{proof}

\vspace{5mm}

We show the following variant of \cite{naber-valtorta:varifold} (recall that $r_{1,1}(\del\Omega) = 1/\Gamma$).
\begin{theorem}[compare from \cite{naber-valtorta:varifold}]\label{thm:fb-nv}
There is an $\eps_3 = \eps_3(n, \Lambda)$, so that if $T \in \cI_n(B_2)$ is area-minimizing, with free-boundary in $\Omega$, and $||T|| \leq \Lambda$, and $\Gamma \leq \eps_3$, then
\[
\haus^{n-7}(\sing T \cap B_1) \leq c(n, \Lambda).
\]
\end{theorem}
When $T = \del[E]$, we get $\Lambda = c(n)$, and then using the decomposition \eqref{eqn:fb-inside-decomp} in an identical argument to Lemma \ref{lem:local-bound}, we deduce the required \eqref{eqn:fb-b-est}.

The argument of \cite{naber-valtorta:varifold} requires only the monotonicity formula \eqref{eqn:fb-mono}, and the following two theorems, which are essentially Lemmas 7.2, 7.3 and Theorem 6.1 in \cite{naber-valtorta:varifold} (or Lemma 3.1, Theorem 5.1 in \cite{me-engelstein:fb}).  The rest of \cite{naber-valtorta:varifold} is entirely general (see e.g. \cite{me:general-nv}).

\begin{theorem}\label{thm:nv-dich}
There is an $\eta_0 = \eta_0(n, \alpha, \Lambda, \gamma, \rho)$, so that the following holds.  Take $B_{6r}(x) \subset B_2$. Let $T \in \cI_n(B_{6r}(x))$ be an area-minimizer with free-boundary in $\Omega$, and take $\eta \leq \eta_0$.  Suppose
\[
\tilde\theta_T(x, 6r) \leq \Lambda, \quad \Gamma \leq \eta, \quad \sup_{B_{3r}(x)} \tilde\theta_T(z, 3r) \leq E,
\]
then at least one of the following occurs:
\begin{enumerate}
\item we have
\[
\sing T \cap B_r(x) \subset \{ z \in B_r(x) : \tilde\theta_T(z, \gamma \rho r) \geq E - \gamma \}, \quad \text{or}
\]
\item there is an affine $(n-8)$-space $p + L^{n-8}$, so that 
\[
\{ z \in B_r(x) : \tilde\theta_T(z, 3\eta r) \geq E - \eta \} \subset B_{\rho r}(p + L) . 
\]
\end{enumerate}
\end{theorem}

\begin{theorem}\label{thm:nv-beta}
There is a $\delta(n, \alpha, \Lambda)$ so that the following holds.  Take $B_{10r}(x) \subset B_2$.  Let $T \in \cI_n(B_{10 r}(x))$ be an area-minmizer with free-boundary in $\Omega$, and $\mu$ a finite Borel measure.  Suppose that
\[
\tilde\theta_T(x, 10r) \leq \Lambda, \quad \Gamma \leq \delta, \quad \tilde\theta_T(x, 8r) - \tilde\theta_T(x, \delta r) < \delta, \quad x \in \sing(T) .
\]
Then we have
\begin{align*}
&\inf_{p + L^{n-7}} \frac{1}{r^{n-5}}\int_{B_r(x)} \dist(z, p + L)^2 d\mu(z) \\
&\quad \leq \frac{c(n, \alpha, \Lambda)}{r^{n-7}} \int_{B_r(x)} \tilde\theta_T(z, 8r) - \tilde\theta_T(z, r) + c(n)\Gamma r d||T||(z),
\end{align*}
where the infimum is over affine $(n-7)$-planes $p + L^{n-7}$.
\end{theorem}

\begin{proof}[Proof of Theorem \ref{thm:nv-dich}]
The proof consists of two contradiction arguments, verbatim to Theorem 5.1 in \cite{me-engelstein:fb}.  In place of the $\eps$-strata, we use the following consequence of Lemma \ref{lem:fb-conv}: Suppose $T_i \in \cI_n(B_6)$ is a sequence of area-minimizers with free-boundary in $\Omega_i$, so that $r_{1,1}(\del\Omega_i \cap B_6) \to \infty$ and $T_i \to T$.  If $T' \llcorner B_{2}$ coincides with a cone, that is invariant along an $(n-6)$-space, then $T_i \llcorner B_1$ is regular for sufficiently large $i$.
\end{proof}

\begin{proof}[Proof of Theorem \ref{thm:nv-beta}]
The proof divides into two parts, which are verbatim to Lemma 6.2 and Proposition 6.6 in \cite{naber-valtorta:varifold} (or Theorem 5.1 in \cite{me-engelstein:fb}).  The first part is a direct consequence of the monotonicity formula \eqref{eqn:fb-mono}.  The second part is a straightforward contradiction argument.  The proof in \cite{naber-valtorta:varifold} uses varifold convergence.  For integral currents, one can use the fact for any $(n+1)$-form $\omega$ and any vector $v$, we have
\[
<\vec{T}, \omega \llcorner v> = <v \wedge \vec{T}, \omega> \quad \text{and} \quad |v \wedge \vec{T}| = |\pi_{T^\perp}(v)| . \qedhere
\]
\end{proof}

\bibliographystyle{alpha}
\bibliography{references}

\end{document}